\documentclass[ejs,preprint,dvips,twoside]{imsart}

\RequirePackage[OT1]{fontenc}
\RequirePackage{amsthm,amsmath,amssymb}
\RequirePackage[numbers]{natbib}
\RequirePackage[colorlinks,citecolor=blue,urlcolor=blue]{hyperref}
\RequirePackage{hypernat}

\doi{10.1214/11-EJS601}
\pubyear{2011}
\volume{5}
\issue{0}
\firstpage{127}
\lastpage{145}

\startlocaldefs
\def\E{\mathbb{E}}
\def\N{\mathbb{N}}
\def\R{\mathbb{R}}

\def\E{\mathbb{E}}

\numberwithin{equation}{section}
\theoremstyle{plain}
\newtheorem{thm}{Theorem}[section]
\newtheorem{prop}{Proposition}[section]
\newtheorem{assum}{Assumption}[section]
\newtheorem{lemma}{Lemma}[section]
\endlocaldefs

\begin{document}

\begin{frontmatter}
\title{PAC-Bayesian bounds for sparse~regression estimation
with~exponential~weights}
\runtitle{PAC-Bayesian bounds for sparse regression estimation}

\begin{aug}
\author{\fnms{Pierre} \snm{Alquier}\corref{}\thanksref{t1}\ead[label=e1]{alquier@ensae.fr}
\ead[label=u1,url]{http://alquier.ensae.net/}}

\address{CREST - ENSAE \\
3, avenue Pierre Larousse \\
92240 Malakoff France \\
and \\
 LPMA - Universit\'e Paris 7 \\
 175, rue du Chevaleret \\
75205 Paris Cedex 13 France\\
\printead{e1}\\
\printead{u1}}
\end{aug}
\medskip
\textbf{\and}

\begin{aug}
\author{\fnms{Karim} \snm{Lounici}
\ead[label=e2]{klounici@math.gatech.edu}
\ead[label=u2,url]{http://people.math.gatech.edu/~klounici6/}}

\address{School of Mathematics \\
Georgia Institute of Technology \\
Atlanta, GA 30332-0160 USA \\
\printead{e2}\\
%\printead{u2}
url: \textup{\texttt{\href{http://people.math.gatech.edu/~klounici6/}{http://people.math.gatech.edu/$\sim$klounici6/}}
}}

\thankstext{t1}{Research partially supported by the French ``Agence
Nationale pour la Recherche''
under grant ANR-09-BLAN-0128 ``PARCIMONIE''.}

\runauthor{P. Alquier and K. Lounici}

\affiliation{Univ. Paris 7, CREST and GeorgiaTech}

\end{aug}

\begin{abstract}
We consider the sparse regression model where the number of
parameters $p$ is larger than the sample size $n$. The difficulty
when considering high-dimensional problems is to propose estimators
achieving a good compromise between statistical and computational
performances. The Lasso is solution of a convex minimization
problem, hence computable for large value of $p$. However stringent
conditions on the design are required to establish fast rates of
convergence for this estimator. Dalalyan and Tsybakov
\cite{arnak,Dal-Tsy-2010,Dal-Tsy-2010b} proposed an exponential
weights procedure achieving a good compromise between the
statistical and computational aspects. This estimator can be
computed for reasonably large $p$ and satisfies a sparsity oracle
inequality in expectation for the empirical excess risk only under
mild assumptions on the design. In this paper, we propose an
exponential weights estimator similar to that of \cite{arnak} but
with improved statistical performances. Our main result is a
sparsity oracle inequality in probability for the true excess
risk.
\end{abstract}

\begin{keyword}[class=AMS]
\kwd[Primary ]{62J07}
\kwd[; secondary ]{62J05}
\kwd{62G08}
\kwd{62F15}
\kwd{62B10}
\kwd{68T05}
\end{keyword}

\begin{keyword}
\kwd{Sparsity oracle inequality}
\kwd{high-dimensional regression}
\kwd{exponential weights}
\kwd{PAC-Bayesian inequalities}
\end{keyword}
\received{\smonth{9} \syear{2010}}
\end{frontmatter}

\tableofcontents

%s1 ###
\section{Introduction}

We observe $n$ independent pairs $(X_{1},Y_{1}),\ldots,(X_{n},Y_{n})\in
\mathcal{X}\times \mathbb R$ (where $\mathcal{X}$ is any measurable
set) such that
%e1 ###
\begin{equation}\label{regmodel}
  Y_{i} = f(X_{i}) + W_{i},\, 1\leqslant i \leqslant n,
\end{equation}
where $f\,:\, \mathcal{X}\rightarrow \mathbb{R}$ is the unknown
regression function and the noise variables $W_1,\ldots,W_n$ are
independent of the design $(X_1,\ldots,X_n)$, satisfy also $\mathbb
E W_i = 0$ and $\mathbb E W_i^2\leqslant \sigma^2$ for any
$1\leqslant i \leqslant n$ and for some known $\sigma^2>0$. The
distribution of the sample is denoted by $\mathbb P$, the
corresponding expectation is denoted by $\mathbb E$. For any
function $g\,:\, \mathcal{X}\rightarrow \mathbb{R}$ define $\|g
\|_{n} = \left(\sum_{i=1}^{n}g(X_{i})^{2}/n
 \right)^{1/2}$ and  $\|g\| =
\left(\mathbb{E}\|g \|_{n}^{2}\right)^{1/2}$. Let $\mathcal{F} =
\{\phi_{1},\ldots,\phi_{p}\}$ be a set---called dictionary---of
functions $\phi_{j}\,:\, \mathcal{X} \rightarrow \mathbb{R}$ such
that $\|\phi_j\| =1$ for any $j$ (this assumption can be relaxed).
For any $\theta\in \R^{p}$ define $f_{\theta} =
\sum_{j=1}^{p}\theta_{j}\phi_{j}$, the empirical risk
$$
r(\theta) = \frac{1}{n}\sum_{i=1}^{n}
\Bigl(Y_{i}-f_{\theta}(X_{i})\Bigr)^{2},
$$
and the integrated risk
$$
R(\theta) = \mathbb E\left[\frac{1}{n}\sum_{i=1}^{n}
\Bigl(Y_{i}'-f_{\theta}(X_{i}')\Bigr)^{2}\right],
$$
where $\{(X_1',Y_1'),\ldots,(X_n',Y_n')\}$ is an independent
replication of $\{(X_{1},Y_{1}),\ldots,\break (X_{n},Y_{n})\}$. Let us choose
$ \overline{\theta}\in\arg\min_{\theta\in\mathbb{R}^{p}} R(\theta)
$. Note that the minimum may not be unique, however we do not need
to treat the identifiability question since we consider in this
paper the prediction problem, i.e., find an estimator $\hat\theta_n$
such that $R(\hat\theta_n)$ is close to $\min_{\theta\in
\R^p}R(\theta)$ up to a positive remainder term as small as
possible.

It is a known fact that the least-square estimator
$\hat{\theta}^{LSE}_n \in \arg\min_{\theta\in\R^p} r(\theta)$
performs poorly in high-dimension $p>n$. Indeed, consider for
instance the deterministic design case with i.i.d. noise variables
$N(0,\sigma^{2})$ and a full-rank design matrix, then
$\hat{\theta}^{LSE}$ satisfies
$$\mathbb E \left[\|f_{\hat{\theta}^{LSE}_n} - f\|_n^2\right] - \|f_{\overline{\theta}} - f\|^2_n = \sigma^{2}.$$
In the same context, assume now there exists a vector
$\overline{\theta}\in\arg\min_{\theta \in\mathbb R^p} R(\theta)$
with a number of nonzero coordinates $p_0 \leq n$. If the indices of
these coordinates are known, then we can construct an estimator
$\hat{\theta}^{0}_n$ such that
\vspace*{-3pt}
$$
\mathbb E \left[\|f_{\hat{\theta}^{0}_n} - f\|_n^2 \right]-
\|f_{\overline{\theta}} - f\|^2_n = \sigma^2\frac{p_0}{n}.
\vspace*{-3pt}
$$
The estimator $\hat{\theta}^{0}_n$ is called oracle estimator since
the set of indices of the nonzero coordinates of $\overline{\theta}$
is unknown in practice. The issue is now to build an
estimator, when the set of nonzero coordinates of
$\overline{\theta}$ is unknown, with statistical performances close
to that of the oracle estimator $\hat{\theta}^{0}_n$.

A possible approach is to consider solutions of penalized empirical
risk minimization problems:
\vspace*{-3pt}
$$ \hat{\theta}_{\mathrm{pen}} \in \arg\min_{\theta\in\R^p} \left\{\frac{1}{n}\sum_{i=1}^{n}
\Bigl(Y_{i}-f_{\theta}(X_{i})\Bigr)^{2} + \mathrm{pen}(\theta)
\right\},
\vspace*{-3pt}
$$
 where the penalization $pen(\theta)$ is proportional to
the number of nonzero components of $\theta$ such as for instance
AIC, $C_{p}$ and BIC criteria \cite{aic,cp,bic}. Bunea, Tsybakov and
Wegkamp \cite{BTW07} established for the BIC estimator
$\hat{\theta}_n^{BIC}$ the following non-asymptotic sparsity oracle
inequality. For any $\epsilon>0$ there exists a constant
$C(\epsilon)>0$ such
that for any $p\geqslant 2, n\geqslant 1$ we have%($\mathrm{pen}(\theta) = \lambda \mathrm{card}(\{j:\theta_j\neq 0\})$, $\lambda = A\sigma\sqrt{(\log p) /n}$, $A>0$)
\vspace*{-3pt}
$$
\mathbb E \left[\|f_{\hat{\theta}^{BIC}_n} - f\|_n^2\right]
\leqslant (1+\epsilon) \|f_{\overline{\theta}} - f\|^2_n +
C(\epsilon)\sigma^{2}\frac{p_0}{n}\log \left(\frac{ep}{p_0\vee
1}\right).
\vspace*{-3pt}
$$
Despite good statistical properties, these estimators can only be
computed in practice for $p$ of the order at most a few tens since
they are solutions of non-convex combinatorial optimization problems.

Considering convex penalty function leads to computationally
feasible optimization problems. A popular example of convex
optimization problem is the Lasso estimator (cf. Frank and Friedman
\cite{bridge}, Tibshirani \cite{LassoTib}, and the parallel work of
Chen, Donoho and Saunders \cite{basis} on basis pursuit) with the
penalty term $\mathrm{pen}(\theta) = \lambda |\theta|_1$, where
$\lambda>0$ is some regularization parameter and, for any integer
$d\geq 2$, real $q>0$ and vector $z\in \mathbb R^d$ we define $|z|_q
= (\sum_{j=1}^{d}|z_j^q|)^{1/q}$ and $|z|_\infty = \max_{1\leq j
\leq d} |z_j|$. Several algorithms allow to compute the Lasso for
very large $p$, one of the most popular is known as LARS, introduced
by Efron, Hastie, Johnstone and Tibshirani \cite{LARS}.
However, the Lasso estimator requires strong assumptions on the
matrix $A=(\phi_{j}(X_i))_{1\leqslant i \leqslant n,1\leqslant j
\leqslant p}$ to establish fast rates of convergence results. Bunea,
Tsybakov and Wegkamp \cite{Lasso2} and Lounici \cite{lounici-ejs}
assume a mutual coherence condition on the dictionary. Bickel, Ritov
and Tsybakov \cite{Lasso3} and Koltchinskii \cite{kolt} established
sparsity oracle inequalities for the Lasso under a restricted
eigenvalue condition. Cand\`es and Tao \cite{Dantzig} and
Koltchinskii \cite{koltbernou} studied the Dantzig Selector which is
related to the Lasso estimator and suffers from the same
restrictions. See, e.g., Bickel, Ritov and Tsybakov \cite{Lasso3}
for more details. Several alternative penalties were recently
considered. Zou \cite{zou-adaptive} proposed the adaptive Lasso
which is the solution of a penalized empirical risk minimization
problem with the penalty $\mathrm{pen}(\theta) = \lambda
\sum_{j=1}^p\frac{1}{|\hat{w}_j|}|\theta_j|$ where $\hat{w}$ is an
initial estimator of $\theta$. Zou and Hastie \cite{zou-elastic}
proposed the elastic net with the penalty $\mathrm{pen}(\theta) =
\lambda_1 |\theta|_1 + \lambda_2|\theta|_2^2$,
$\lambda_1,\lambda_2>0$. Meinshausen and B\"uhlmann \cite{stability}
and Bach \cite{bach09} considered bootstrapped Lasso. See also Ghosh
\cite{aen} or Cai, Xu and Zhang \cite{cai} for more alternatives to
the Lasso. All these methods were motivated by their superior
performances over the Lasso either from the theoretical or the
practical point of view. However, strong assumptions on the design
are still required to establish the statistical properties of these
methods (when such results exist). A recent paper by van de Geer and
B\"uhlmann \cite{vdgb} provides a complete survey and comparison of
all these assumptions.

Simultaneously, the PAC-Bayesian approach for regression estimation
was developed by Audibert \cite{AudibertReg,Audibert} and Alquier
\cite{Alquier2006,AlquierPAC}, based on previous works in the
classification context by Catoni \cite{Classif,Cat7,manuscrit}, Mc
Allester \cite{McAllester}, Shawe-Taylor and Williamson
\cite{STW97}, see also Zhang \cite{zhang} in the context of density
estimation. This framework is very well adapted for studying the excess
risk $R(\cdot) - R(\overline{\theta})$ in the regression context
since it requires very weak conditions on the dictionary. However, the methods
of these papers are not
designed to cover the high-dimensional setting under the sparsity assumption.
 Dalalyan and Tsybakov
\cite{daltsy07,arnak,Dal-Tsy-2010,Dal-Tsy-2010b} propose an
exponential weights procedure related to the PAC-Bayesian approach
with good statistical and computational performances. However they
consider deterministic design, establishing their
statistical result only for the empirical excess risk instead of the
true excess risk $R(\cdot)-R(\overline{\theta})$.

In this paper, we propose to study two exponential weights
estimation procedures. The first one is an exponential weights
combination of the least squares estimators in all the possible
sub-models. This estimator was initially proposed by Leung and
Barron \cite{leung} in the deterministic design setting. Note that
in the literature on aggregation and exponential weights, the
elements of the dictionary are often arbitrary preliminary
estimators computed from a frozen fraction of the initial sample so
that these estimators are considered as deterministic functions, the
aggregate is then computed using this dictionary and the remaining
data. This scheme is referred to as 'data splitting'. See for
instance Dalalyan and Tsybakov \cite{Dal-Tsy-2010} and Yang
\cite{yangbernoulli}. Leung and Barron \cite{leung} proved that data
splitting is not necessary in order to aggregate least squares
estimators and raised the question of computation of this estimator
in high dimension. In this paper we explicit the oracle inequality
satisfied by this estimator in the high-dimensional case. For the
second procedure, the design may be either random or deterministic.
We adapt to the regression framework PAC-Bayesian techniques
developed by Catoni \cite{manuscrit} in the classification framework
to build an estimator satisfying a sparsity oracle inequality for
the true excess risk. Even though we do not study the computational
aspect in this paper, it should be noted that efficient Monte Carlo
algorithms are available to compute these exponential weights
estimators for reasonably large dimension $p$ ($p\simeq 5000$), see
in particular the monograph of Marin and Robert \cite{Christian2}
for an introduction to MCMC methods. Note also that in a work
parallel to ours, Rigollet and Tsybakov \cite{RigTsy10} consider
also an exponential weights procedure with discrete priors and
suggest a version of the Metropolis-Hastings algorithm to compute
it.
\vfill\eject

%our
%estimators are clearly related to Bayesian estimators. So it is
%possible to use the computational Bayesian theory, see
%the monograph  for an
%introduction to Monte Carlo algorithms in Bayesian theory. More
%specifically, the Bayesian point of view for the variable selection
%problem was considered in several papers: George \cite{george2},
%George and McCulloch \cite{george1}, West \cite{west}, Jiang
%\cite{jiang}, Cui and George \cite{george3}, Bogdan {\it et al}
%\cite{ghosh2008}, Liang {\it et al} \cite{liangetal}, Scott and
%Berger \cite{scottberger} among others. See in particular
%\cite{george1,nott} for the algorithmic aspects of Monte Carlo
%techniques. It may also be possible to use the RJMCMC
%method proposed by Green \cite{RJMCMC} to compute our second
%estimator.

The paper is organized as follows. In Section \ref{part-K} we define
an exponential weights procedure and derive a sparsity oracle
inequality in expectation when the design is deterministic. In
Section \ref{estimator}, the design can be either deterministic or
random. We propose a modification of the first exponential weights
procedure for which we can establish a sparsity oracle inequality in
probability for the true excess risk. Finally Section \ref{proofs}
contains all the proofs of our results.

%s2 ###
\section{Sparsity oracle inequality in expectation}
\label{part-K}
%For the sake of simplicity, we also assume in this section that the
%elements of $\mathcal{D}$ are normalized, that is
%$\|\phi_{j}(.)\|=1$ for any $j$.

Throughout this section, we assume that the design is deterministic
and the noise variables $W_1,\ldots,W_n$ are i.i.d. gaussian
$N(0,\sigma^2)$.

For any $J\subset\{1,\ldots,p\}$ and $K>0$ define
%e2 ###
\begin{equation}
\Theta(J) = \biggl\{\theta\in\mathbb{R}^{p}:
         \quad \forall j\notin J,\quad \theta_{j} = 0\biggr\},
\end{equation}
and
%e3 ###
\begin{equation}
\Theta_{K}(J) = \biggl\{\theta\in\mathbb{R}^{p}: \quad |\theta|_{1}
\leq K \quad \text{and} \quad  \forall j\notin J,\quad\theta_{j} = 0 \biggr\}.
\end{equation}
For the sake of simplicity we will write
$\Theta_{K}=\Theta_{K}(\{1,\ldots,p\})$.

%
%We consider the following least square estimator of $f$: namely,
%$f_{\hat{\theta}_{J}}$ where $\hat{\theta}_{J}\in\arg\min_{\theta
%\in \Theta(J)}r(\theta)$.
For any subset $J\subset\{1,\ldots,p\}$ define
$$ \hat{\theta}_{J} \in \arg\min_{\theta\in\Theta(J)} r(\theta), $$
where $r(\theta)=\frac{1}{n}\sum_{i=1}^{n}
(Y_{i}-f_{\theta}(X_{i}))^{2}= \|Y - f_{\theta}\|_{n}^{2}
$ with $Y=(Y_{1},\ldots,Y_{n})^{T}$. Denote by
$\mathcal{P}_n(\{1,\ldots,p\})$ the set of all subsets of
$\{1,\ldots,p\}$ containing at most $n$ elements. The aggregate
$\hat{f}_n$ is defined as follows
%e4 ###
\begin{equation}\label{aggregate}
\hat{f}_n = f_{\hat{\theta}_n},\quad\hat{\theta}_{n} =
\hat{\theta}_{n}(\lambda,\pi) \stackrel{\triangle}{=} \frac{ \sum_{
J \in \mathcal{P}_n(\{1,\ldots,p\})} \pi_{J} e^{-\lambda\bigl(
r(\hat{\theta}_{J}) + \frac{2\sigma^{2}|J|}{n} \bigr)}
\hat{\theta}_{J} } {\sum_{ J \in \mathcal{P}_n(\{1,\ldots,p\})}
\pi_{J} e^{-\lambda \bigl(r(\hat{\theta}_{J}) +
\frac{2\sigma^{2}|J|}{n}\bigr)} }
\end{equation}
where $\lambda>0$ is the temperature parameter, $\pi $ is the prior
probability distribution on $\mathcal{P}(\{1,\ldots,p\})$, the set
of all subsets of $\{1,\ldots,p\}$, that is, for any $J\in
\{1,\ldots,p\}$, $\pi_{J}\geq 0$ and $\sum_{J\in
\mathcal{P}(\{1,\ldots,p\})}\pi_{J} = 1$.

The next result is a reformulation in our context of Theorem 8 of
\cite{leung}.
\begin{prop}
\label{thm_K} Assume that the noise variables $W_{1},\ldots,W_{n}$
are i.i.d. $N(0,\sigma^2)$. Then the aggregate $\hat{\theta}_{n}$
defined by (\ref{aggregate}) with $0<\lambda \leqslant
\frac{n}{4\sigma^{2}}$ satisfies
%e5 ###
\begin{equation}\label{theo1}
 \mathbb{E}\left[r(\hat{\theta}_{n})  \right] \leqslant
\min_{ J\in \mathcal{P}_n(\{1,\ldots,p\})
}\left\{\mathbb{E}[r(\hat{\theta}_{J})] +
\frac{1}{\lambda}\log\left( \frac{1}{\pi_{J}} \right)\right\}.
\end{equation}
\end{prop}

Proposition \ref{thm_K} holds true for any prior $\pi$. We suggest
using the following prior. Fix $\alpha\in (0,1)$ and define
%e6 ###
\begin{equation}\label{prior}
   \pi_{J} = \frac{\alpha^{|J|}}{\sum_{j=0}^{n}\alpha^{j}} {p \choose
   |J|}^{-1},\quad \forall J\in \mathcal{P}_n(\{1,\ldots,p\}), \text{ and
   }
   \pi_{J}=0 \text{ if } |J|>n.
\end{equation}
As a consequence, we obtain the following immediate corollary of
Proposition~\ref{thm_K}.
\begin{thm}\label{thm_K2}
Assume that the noise variables $W_{1},\ldots,W_{n}$ are i.i.d.
$N(0,\sigma^2)$. Then the aggregate
$\hat{f}_{n}=f_{\hat{\theta}_{n}}$, with $\lambda =
\frac{n}{4\sigma^{2}}$ and $\pi$ taken as in (\ref{prior}),
satisfies
\begin{multline}\label{result-K}
\mathbb{E}\left[ \|\hat{f}_{n}-f\|_{n}^{2} \right]
\\
 \leqslant
\min_{\theta\in \mathbb R^p} \left\{ \|f_{\theta}-f\|_{n}^{2} +
\frac{\sigma^{2}|J(\theta)|}{n}\left(4\log\left( \frac{p
e}{|J(\theta)|\alpha}\right)+1\right)
 + \frac{4\sigma^{2}\log\bigl(\frac{1}{1-\alpha}\bigr)}{n}\right\},
\end{multline}
where for any $\theta\in \mathbb R^p$ $J(\theta) = \{j\,:\,\theta_j
\neq 0 \}$.
\end{thm}

This result improves upon \cite{arnak} which established in Theorem
6 for gaussian noise and deterministic design
\begin{multline*}
  \mathbb E \left[ \|\hat f _n - f\|_n^2 \right]
\\
\leq \min_{\theta\in \R^p}\left\{\|f_\theta - f
  \|_n^2 + \frac{16\sigma^2|J(\theta)|}{n}\left( 1+
  \log_+\left(\frac{\sqrt{\mathrm{Tr}(A^\top A)}}{M(\theta)\sigma}|\theta|_1
  \right)\right) + \frac{\sigma^2}{n}\right\},
\end{multline*}
where $\log_+ x = \max \left\{\log x, 0\right\}$ and we recall that
$A = (\phi_{j}(X_i))_{1\leq i \leq n,1 \leq j \leq p}$. Note that
our bound is faster by a factor $\log_+ |\theta|_1$. Note also that
the bound in the above display grows worse for large values of
$|\theta|_1$.

In order to evaluate the performance of these exponential weights
procedures, \cite{Tsy2003} developed a notion of optimal rate of
sparse prediction. In particular, \cite{RigTsy10} established that
there exists a numerical constant $c^*>0$ such that for all
estimator $T_n$
\[
%\begin{multline*}
\sup_{\substack{\theta\in \R^{p}\setminus \{0\}:\\
|J(\theta)|\leq
s}}
  \sup_{f}\left\{ \mathbb E\left[\|T_n - f \|_n^2\right] - \|f_\theta - f \|_n^2
  \right\}\geq c^* \frac{\sigma^2}{n}\left[\mathrm{rank}(A)\wedge
  s\log\left( 1 + \frac{ep}{s}\right)\right].
%\end{multline*}
\]

The above display combined with Theorem \ref{thm_K2} shows that the
exponential weights procedure (\ref{aggregate}) with the prior
(\ref{prior}) achieves the optimal rate of sparse prediction for any
vector $\theta$ satisfying $|J(\theta)|\leq
\frac{\mathrm{rank}(A)}{\log(1+ep)}$.

%s3 ###
\section{Sparsity oracle inequality in probability}\label{estimator}

In Section \ref{part-K} we assumed the design is deterministic and
we established an oracle inequality in expectation with the optimal
rate of sparse prediction. We want now to establish an oracle
inequality in probability that holds true for deterministic and
random design.

{\bf From now on, the design can be either deterministic or random}.
We make the following mild assumption:
$$L = \max_{1\leq j \leq M}\|\phi_j\|_{\infty} <\infty.$$

We assume in this section that the noise variables are subgaussian.
More precisely we have the following condition.
\begin{assum}\label{bruit-Pac}
The noise variables $W_1,\ldots,W_n$ are independent and independent
of $X_1,\ldots,X_{n}$. We assume also that there exist two known
constants $\sigma>0$ and $\xi>0$ such that
$$ \mathbb{E}(W_{i}^{2})\leq \sigma^{2} $$
$$ \forall k\geq 3,\quad \mathbb{E}(|W_{i}|^{k}) \leq \sigma^{2} k! \xi^{k-2}.$$
\end{assum}

%\begin{rmk}
%This is for example the case if
%$\mathcal{X}=\left\{x=(x_{1},\ldots,x_{p})' \in\mathbb{R}^{p}:
%\|x\|\leq 1\right\}$ and $\phi_{j}(x)=x_{j}$. In this case
%$$ f_{\theta}(x) = \sum_{j=1}^{p} \theta_{j} x_{j} = \left<\theta,x\right>.$$
%\end{rmk}

The estimation method is a version of the Gibbs estimator introduced
in \cite{Cat7,manuscrit}. Fix $K\geq 1$. First we
define the prior probability distribution as follows. For any
$J\subset\{1,\ldots,p\}$ let ${\rm u}_{J}$ denote the uniform measure
on $\Theta_{K+1}(J)$. We define
$$ {\rm m}(d\theta) =
\sum_{ J\subset\{1,\ldots,p\}} \pi_{J} {\rm u}_{J}(d\theta)
$$
with $\pi$ taken as in (\ref{prior}).

We are now ready to define our estimator. For any $\lambda>0$ we
consider the probability measure $\tilde{\rho}_\lambda$ admitting
the following density w.r.t. the probability measure ${\rm m}$
%e7 ###
\begin{equation}\label{Gibbsmeasure}
\frac{d\tilde{\rho}_{\lambda}}{d{\rm m}}(\theta) = \frac{e^{-\lambda
r(\theta)}} {\int_{\Theta_{K}}e^{-\lambda r} d{\rm m}}.
\end{equation}
The aggregate $\tilde{f}_n$ is defined as follows
%e8 ###
\begin{equation}\label{Estimateur_Pierre}
\tilde{f}_n = f_{\tilde{\theta}_n},\quad \tilde{\theta}_{n} =
\tilde{\theta}_{n}(\lambda,{\rm m}) = \int_{\Theta_{K}}\theta
\tilde{\rho}_{\lambda}(d\theta).
\end{equation}
%The practical computation of $\tilde{\theta}_{n}$ is discussed in
%Section 4.

Define $$\mathcal{C}_1 = \left[8\sigma^{2} + (2\|f\|_\infty +
L(2K+1))^2\right]\vee \left[ 8[\xi + (2\|f\|_\infty +
L(2K+1))]L(2K+1) \right].$$
%and $$\mathcal{C}_2 = L[2\|f\|_{\infty}
%+ L(2K+1) + 2\sigma].$$
We can now state the main result of this section.
\begin{thm}
\label{SOI} Let Assumption \ref{bruit-Pac} be satisfied. Take $K>1$
and $\lambda = \lambda^* = \frac{n}{2\mathcal{C}_1}$. Assume that
$\arg\min_{\theta\in \mathbb R^p}R(\theta) \cap \Theta_K \neq
\emptyset$.
%For any vector $\overline{\theta}\in \arg\min_{\mathbb R^p}R \cap
%\Theta_K$
Then we have, for any $\varepsilon\in (0,1)$ and any
$\bar{\theta}\in \arg\min_{\theta\in \mathbb R^p}R(\theta)\cap
\Theta_K$, with probability at least $1 - \epsilon$,
\begin{multline*}
R(\tilde{\theta}_{n}) \leq R(\bar{\theta}) + \frac{3L^2}{n^2} +
\frac{8\mathcal{C}_1}{n}\Biggl[
|J(\bar{\theta})|\log\left(K+1\right)
 \\+
\left(|J(\bar{\theta})|\log\left(\frac{enp}{\alpha
|J(\bar{\theta})|} \right) +
\log\left(\frac{2}{\varepsilon(1-\alpha)}\right)\right) \Biggr].
\end{multline*}
\end{thm}

The choice $\lambda=\lambda^{*}$ comes from the optimization of a
(rather pessimistic) upper bound on the risk $R$ (see Inequality
(\ref{boundlambda}) in the proof of this theorem, page
\pageref{boundlambda}). However this choice is not necessarily the
best choice in practice even though it gives the good order of
magnitude for $\lambda$. The practitioner may use cross-validation
to properly tune the temperature parameter.
\medskip

Theorem \ref{SOI} %2
 improves upon previous results on the following points:
\medskip
\vspace*{-3pt}

1) Our oracle inequality is sharp in the sense that the leading
factor in front of $R(\bar{\theta})$ is equal to $1$ and we require
only $\max_j \|\phi_j\| <\infty$ whereas $l_1$-penalized empirical
risk minimization procedures such as Lasso or Dantzig selector have
a leading factor strictly larger than $1$ and impose in addition
stringent conditions on the dictionary (cf. \cite{Lasso3,Dantzig}).
For instance, \cite{Lasso3} imposes the design matrix $A =
(\phi_{j}(X_i))_{1\leq i \leq n,1 \leq j \leq p}$ to be
deterministic and to satisfy the following restricted eigenvalue
condition for the Lasso
\begin{equation*}\label{RE}
\kappa(s)\stackrel{\triangle}{=}\min_{\substack{J_{0}\subset
\{1,\ldots,p\}:\\
|J_{0}|\leqslant
s}}\hspace{0.25cm}\min_{\substack{\Delta\in \R^p\setminus \{0\}:\\
|\Delta_{J_{0}^{c}}|_{1}\leqslant
3|\Delta_{J_{0}}|_{1}}}\frac{|A\Delta|_{2}}{\sqrt{n}|\Delta_{J_{0}}|_{2}}>0,
\end{equation*}
where for any $\Delta\in \R^{p}$ and $J \subset \{1,\ldots,p\}$, we
denote by $\Delta_J$ the vector in $\R^{p}$ which has the same
components as $\Delta$ on $J$ and zero coordinates on the complement
$J^c$. Assuming in addition that the noise is gaussian
$N(0,\sigma^2)$ and taking $\lambda = A\sigma \sqrt{\frac{\log
p}{n}}$, $A>8$, \cite{Lasso3} proved that the Lasso $\hat\theta^L$
satisfies with probability at least $1 - M^{1-A^2/8}$
\begin{equation*}
\frac{1}{n}\|f_{\hat{\theta}^L} - f\|_n^2 \leq
(1+\eta)\frac{1}{n}\|f_{\theta} - f\|_n^2 + C(\eta)
A^2\sigma^2|J(\theta)|\frac{\log p}{n}, \quad \forall \theta\in
\R^p,
\end{equation*}
where $\eta,C(\eta)>0$ and $C(\eta)$ increases to $+\infty$ as
$\eta$ tends to $0$.

On the downside, our estimator requires the additional condition
$|\bar \theta|_1 \leq K$. This condition is common in the
PAC-bayesian literature. Removing this condition is a difficult
problem and does not seem possible with the actual techniques of
proof where this condition is needed in order to apply Bernstein's
inequality.

\medskip
\vspace*{-3pt}

2) We establish a sparsity oracle inequality in probability for the
integrated risk $R(\cdot)$ whereas previous results on the
exponential weights are given in expectation
\cite{daltsy07,arnak,Dal-Tsy-2010b,JRT08,lounici-07,RigTsy10}.

\medskip
\vspace*{-3pt}

3) Unlike mirror averaging or progressive mixture rules, satisfying
similar inequalities in expectation, our estimator does not involve
an averaging step \cite{Dal-Tsy-2010,JRT08,lounici-07}. As a
consequence, its computational complexity is significantly reduced
as compared to those procedures with averaging step. For instance
\cite{lounici-07} considered the model (\ref{regmodel}) with random
design and i.i.d. observations $(X_1,Y_1),\ldots,(X_n,Y_n)$, $n\geq
2$. The studied estimator is the following mirror averaging scheme
$$
f_{\hat\theta^{MA}},\quad \theta^{MA}  =
\frac{1}{n}\sum_{k=0}^{n-1}\tilde{\theta}_{k},
$$
where $\tilde \theta_0 = \int_{\Theta(K)}\theta d\Pi$ and for any
$1\leq k\leq n-1$, $\tilde{\theta}_k$ is defined similarly as
$\tilde\theta_n$ in (\ref{Gibbsmeasure})--(\ref{Estimateur_Pierre})
with $r(\theta)$ replaced by $r_k(\theta) =\frac{1}{k}
\sum_{i=1}^{k}(Y_i - f_\theta(X_i))^2$. These estimators can be implemented for example by MCMC. In this case, computing the integral
$\int_{\Theta(K)}\theta \tilde\rho_\lambda(d\theta)$ is the most
time-consuming part of the procedure. The procedure
(\ref{Gibbsmeasure})--(\ref{Estimateur_Pierre}) requires computing
this integral only once whereas the mirror averaging procedure of
\cite{lounici-07} requires computing integrals of this form $n$
times.

\medskip
\vspace*{-3pt}

4) Under the assumption $|\bar \theta|_1 \leq K$ for some absolute
constant $K$ and taking $\epsilon = n^{-1}$ we have with probability
at least $1 - n^{-1}$
%e9 ###
\begin{equation}\label{bornesup}
R(\tilde{\theta}_{n}) \leq R(\bar{\theta}) +
C\frac{\mathcal{C}_1}{n}|J(\bar{\theta})|
\log\left(\frac{enp}{\alpha |J(\bar{\theta})|} \right) +
\frac{3L^2}{n^2},
\end{equation}
for some absolute constant $C>0$. In \cite{RigTsy10} a minimax lower
bound in expectation is established for deterministic design and
Gaussian noise. A similar result holds in probability with the same
proof combined with Theorem 2.7 of \cite{tsy09}. There exists
absolute constants $c_1,c_2>0$ such that
%\begin{multline*}
\[
 \inf_{T_n}\sup_{\substack{\theta \in \R^p:\\
 |J(\theta)|\leq s}} \sup_{f}\mathbb P \left[ \|T _n \,{-}\, f\|_n^2  \geq \|f_{\theta }\,{-}\, f
  \|_n^2 + c_1\frac{\sigma^2 }{n}\mathrm{rank}(A)\wedge s\log\left( 1+ \frac{ep}{s}\right)\right] \geq c_2.
%\end{multline*}
\]
If $s\leq \frac{\mathrm{rank}(A)}{\log (1+ep)}$ then we observe that
the upper bound in (\ref{bornesup}) is optimal up to the additional
logarithmic factor $\log n$. Note however that if $p\geq
n^{1+\delta}$ for some $\delta >0$, which is relevant with the
high-dimensional setting we consider in this paper, then our bound
is rate optimal.

\section{Proofs}
\label{proofs}

%s4.1 ###
\subsection{Proofs of Section \ref{part-K}}
\label{proofK}

This proof uses an argument from Leung and Barron \cite{leung}.

\begin{proof}[Proof of Proposition \ref{thm_K}]
The mapping $Y\rightarrow
\hat{f}_{n}(Y)\stackrel{\triangle}{=}(\hat{f}_{n}(X_{1},Y),\ldots,\hat{f}_{n}(X_{n},\break Y))^{T}$
is clearly continuously differentiable by composition of elementary
differentiable functions. For any subset $J\subset\{1,\ldots,p\}$
define $A_{J} = (\phi_{j}(X_i))_{1\le i \le n,j\in J}$, $\Sigma_{J}
= \frac{1}{n}A^{T}_{J}A_J$, $\Phi_J(\cdot) = (\phi_j(\cdot))_{j\in
J}$ and $$g_{J} = e^{ -\lambda\bigl(\|Y-f_{J}\|_{n}^{2} +
\frac{2\sigma^2|J|}{n}\bigr)}$$ where
\begin{equation*}
f_{J}(x,Y)=\frac{1}{n}Y^{T}A_{J}\Sigma_{J}^{+}\Phi_{J}(x)^{T},
\end{equation*}
and $\Sigma_J^+$ denotes the pseudo-inverse of $\Sigma_J$. Denote by
$\partial_i$ the derivative w.r.t. $Y_i$. Simple computations give
\begin{equation*}
\partial_{i}f_{J}(x,Y)
=\frac{1}{n}\Phi_{J}(X_{i})\Sigma_{J}^{+}\Phi_{J}(x)^T,
\end{equation*}
\begin{equation*}
(\partial_{i}f_{J}(X_{1},Y),\ldots,\partial_{i}f_{J}(X_{n},Y))Y=f_{J}(X_{i},Y),
\end{equation*}
and
\begin{eqnarray*}
\sum_{l=1}^{n}f_{J}(X_{l},Y)\partial_{i}f_{J}(X_{l},Y) &=&
f_{J}(X_{i},Y).
\end{eqnarray*}
Thus we have
\begin{eqnarray*}
\partial_{i}(g_{J}) &=& - \lambda\partial_{i}\left( \|Y-f_{J}\|_{n}^{2}
\right)g_{J}\\
&=& -\frac{2\lambda}{n}\left( (Y_{i}-f_{J}(X_{i},Y)) -
\sum_{l=1}^{n}\partial_{i}f_{J}(X_{l},Y)(Y_{l}-f_{J}(X_{l},Y))
 \right)g_{J}\\
 &=&- \frac{2\lambda}{n}(Y_{i}-f_{J}(X_{i},Y))g_{J},
\end{eqnarray*}
%Combining the above equalities we obtain
%\begin{align*}
%\partial_{i}(\theta_{J})&=
%\frac{(\sum_{1\leqslant m' \leqslant n, \, 1 \leqslant k' \leqslant
%\binom{M}{m} }g_{m',k'}\pi_{m',k'})\partial_{i}g_{m,k}
%-g_{m,k}\partial_{i}(\sum_{1\leqslant m' \leqslant n, \, 1 \leqslant
%k' \leqslant \binom{M}{m}}g_{m',k'}\pi_{m',k'})}{(\sum_{1\leqslant
%m' \leqslant n, \, 1 \leqslant k'
%\leqslant \binom{M}{m}}g_{m',k'}\pi_{m',k'})^{2}}\\
%&= -\frac{2}{\beta}\sum_{1\leqslant m' \leqslant n, \, 1 \leqslant
%k' \leqslant
%\binom{M}{m}}(Y_{i}-f_{m',k'}(x_{i}))\theta_{m',k'}f_{m',k'}(x_{i})\pi_{m',k'}\\
%&+ \frac{2}{\beta}\left( \sum_{1\leqslant m' \leqslant n, \, 1
%\leqslant k' \leqslant
%\binom{M}{m}}\theta_{m',k'}f_{m',k'}(x_{i})\pi_{m',k'}\right)\left(\sum_{1\leqslant
%m' \leqslant n, \, 1 \leqslant k' \leqslant
%\binom{M}{m}}(Y_{i}-f_{m',k'}(x_{i}))\theta_{m',k'}\pi_{m',k'}\right),
%\end{align*}
%and
%\begin{equation*}
%\partial_{i}\left(f_{m,k}^{(MC)}(x_{i},\mathbb{Y})\right) = \frac{1}{n}\sum_{l\in
%J_{m,k}}\phi_{l}(x_{i})^{2}.
%\end{equation*}
Recall that $$\hat{f}_n(\cdot,Y)) = \frac{\sum_{J\in
\mathcal{P}_n(\{1,\ldots,p\})}\pi_J g_J f_J(\cdot,Y)}{\sum_{J\in
\mathcal{P}_n(\{1,\ldots,p\})}\pi_J g_J}.$$
We have
%e10 ###
\begin{align}\label{cond-Fubini}
\partial_{i}\hat{f}_{n}(X_{i},Y) =\ &
\frac{\sum_{J\in
\mathcal{P}_n(\{1,\ldots,p\})}\pi_{J}\left(\partial_{i}(g_{J})f_{J}(X_{i},
Y) + g_{J}\partial_{i}(f_{J}(X_{i},Y))\right)}{\sum_{J\in
\mathcal{P}_n(\{1,\ldots,p\})}\pi_J g_J} \nonumber\\
&{}-\frac{\bigl(\sum_{J\in
\mathcal{P}_n(\{1,\ldots,p\})}\pi_J g_J
f_J(X_i,Y)\bigr)\bigl(\sum_{J\in
\mathcal{P}_n(\{1,\ldots,p\})}\pi_J
\partial_i(g_J)\bigr) }{(\sum_{J\in
\mathcal{P}_n(\{1,\ldots,p\})}\pi_J g_J)^2}\nonumber\\
=\ & -\frac{2\lambda}{n}Y_i \hat{f}_n +
\frac{2\lambda}{n}\frac{\sum_{J\in
\mathcal{P}_n(\{1,\ldots,p\})}f_{J}(X_{i},Y)^{2}\pi_J
g_J}{\sum_{J\in
\mathcal{P}_n(\{1,\ldots,p\})}\pi_J g_J}\nonumber\\
&{} %&\hspace{0.5cm}
+ \frac{1}{n}\frac{\sum_{J\in
\mathcal{P}_n(\{1,\ldots,p\})}\Phi_J(X_i)\Sigma_J^{+}\Phi_J(X_i)^T\pi_J
g_J}{\sum_{J\in
\mathcal{P}_n(\{1,\ldots,p\})}\pi_J g_J}\nonumber\\
&{} \frac{2\lambda}{n}Y_i\hat{f}_n(X_i,Y) -\frac{2\lambda}{n}\hat{f}_n^{2}(X_i,Y) \nonumber\\
=\ & \frac{2\lambda}{n}\frac{\sum_{J\in
\mathcal{P}_n(\{1,\ldots,p\})}(f_J(X_i,Y) - \hat{f}_n(X_i,Y))^2\pi_J
g_J}{\sum_{J\in
\mathcal{P}_n(\{1,\ldots,p\})}\pi_J g_J}\nonumber\\
&{} %~&\hspace{0.5cm}
+ \frac{1}{n}\frac{\sum_{J\in
\mathcal{P}_n(\{1,\ldots,p\})}\Phi_J(X_i)\Sigma_J^{+}\Phi_J(X_i)^T\pi_J
g_J }{\sum_{J\in \mathcal{P}_n(\{1,\ldots,p\})}\pi_J g_J}\ge 0.
\end{align}
Consider the following estimator of the risk
%e11 ###
\begin{equation}\label{risk-Stein}
\hat{r}_{n}(Y) =
\|\hat{f}_{n}(Y)-Y\|_{n}^{2}+\frac{2\sigma^{2}}{n}\sum_{i=1}^{n}\partial_{i}\hat{f}_{n}(X_{i},Y)
- \sigma^{2}.
\end{equation}
Using an argument based on Stein's identity as in \cite{LC98} we now
prove that
\begin{equation*}
\E[\hat{r}_{n}(Y)] = \E\left[ \|\hat{f}_{n}(Y) - f \|_{n}^{2}
\right].
\end{equation*}
We have
%e12 ###
\begin{eqnarray}\label{interm1}
\E\left[ \|\hat{f}_{n}(Y)-f\|_{n}^{2} \right] &\!\!=\!\!& \E\left[ \|
\hat{f}_{n}(Y) - Y \|_{n}^{2} +
\frac{2}{n}\sum_{i=1}^{n}W_{i}(\hat{f}_{n}(X_{i},Y)-f(X_{i}))
 \right] - \sigma^{2}\nonumber\\
 &\!\!=\!\!& \E\left[ \| \hat{f}_{n}(Y) - Y \|_{n}^{2} +
\frac{2}{n}\sum_{i=1}^{n}W_{i}\hat{f}_{n}(X_{i},Y)
 \right] - \sigma^{2}.
\end{eqnarray}
For $\mathbf{z}=(z_{1},\ldots,z_{n})^{T}\in\R^{n}$ write
$F_{W,i}(\mathbf{z}) = \prod_{j\neq i}F_{W,i}(z_{j})$, where $F_{W}$
denotes the c.d.f. of the random variable $W_{1}$. Since
$\E(W_{i})=0$ we have
%e13 ###
\begin{eqnarray}\label{interm2}
\E\left[ W_{i}\hat{f}_{n}(X_{i},Y) \right] &\!\!\!\!=\!\!\!\!& \E \!\left[
W_{i}\!\int_{0}^{W_{i}}\! \partial_{i}\hat{f}_{n}(X_{i},Y_{1},\ldots,Y_{i-1},f(X_{i})+z,Y_{i+1},\ldots,Y_{n})dz
\!\right]\nonumber\!\!\!\!\!\\
&\!\!\!\!=\!\!\!\!&\int_{\R^{n-1}}\!\left(
\int_{\R}y\!\int_{0}^{y}\!\partial_{i}\hat{f}_{n}(X_{i},f\,{+}\,\mathbf
{z})dz_{i}dF_{W}(y) \!\right)dF_{W,i}(\mathbf{z}).
\end{eqnarray}
In view of (\ref{cond-Fubini}) we can apply Fubini's Theorem to the
right-hand-side of (\ref{interm2}). We obtain under the assumption
$W \sim \N(0,\sigma^{2})$ that
\begin{eqnarray*}
\int_{\R^{+}}\int_{0}^{y}\partial_{i}\hat{f}_{n}(X_{i},f+\mathbf
{z})dz_{i}dF_{W}(y) &=&
\int_{\R^{+}}\int_{z_{i}}^{\infty}ydF_{W}(y)\partial_{i}\hat{f}_{n}(X_{i},f+\mathbf{z})dz_{i}\\
&=&\int_{\R^{+}}\sigma^{2}\partial_{i}\hat{f}_{n}(X_{i},f+\mathbf{z})dF_{W}(z_{i}),
\end{eqnarray*}
A Similar equality holds for the integral over $\R^{-}$. Thus we
obtain
\begin{equation*}
\E\left[ W_{i}\hat{f}_{n}(X_{i},Y) \right] = \sigma^{2}\E\left[
\partial_{i}\hat{f}_{n}(X_{i},Y)\right].
\end{equation*}
Combining (\ref{risk-Stein}), (\ref{interm1}) and the above display
gives
\begin{equation*}
\E\left[\hat{r}_{n}(Y) \right] = \E\left[ \| \hat{f}_{n}(Y)-f
\|_{n}^{2} \right].
\end{equation*}
Since $ \hat{f}_{n}(\cdot,Y)$ is the expectation of $f_{J}(\cdot,Y)$
w.r.t. the probability distribution $\propto g\cdot \pi$, we have
\begin{equation*}
\| \hat{f}_{n}(\cdot,Y) - Y\|_{n}^{2} = \frac{\sum_{J\in
\mathcal{P}_n(\{1,\ldots,p\})}\bigl(
\|f_{J}(\cdot,Y)\,{-}\,Y\|_{n}^{2}\,{-}\,\|f_{J}(\cdot,Y)\,{-}\,\hat{f}_{n}(Y)
\|_{n}^{2} \bigr)g_{J}\pi_J}{\sum_{J\in
\mathcal{P}_n(\{1,\ldots,p\})}g_{J}\pi_J}.
\end{equation*}
For the sake of simplicity set $f_J = f_J(\cdot,Y)$ and $\hat{f}_n
=\hat{f}_n(\cdot,Y)$. Combining (\ref{risk-Stein}), the above
display and $\lambda\leqslant \frac{n}{4\sigma^{2}}$ yields
\begin{eqnarray*}
\hat{r}_{n}(Y) &=&\frac{\sum_{J\in
\mathcal{P}_n(\{1,\ldots,p\})}\left(
\|f_{J}-Y\|_{n}^{2}+\sum_{i=1}^{n}\bigl(\frac{4\lambda\sigma^2}{n}-1\bigr)\|f_{J}-\hat{f}_{n}\|_n^{2}
\right)g_{J}\pi_J}{\sum_{J\in
\mathcal{P}_n(\{1,\ldots,p\})}\pi_J g_J}\\
&& \hspace{0.5cm}+
\frac{2\sigma^{2}}{n^2}\sum_{i=1}^{n}\frac{\sum_{J\in
\mathcal{P}_n(\{1,\ldots,p\})}\Phi_J(X_i)\Sigma_J^{+}\Phi_J(X_i)^T\pi_J
g_J}{\sum_{J\in \mathcal{P}_n(\{1,\ldots,p\})}\pi_J g_J}  -\sigma^{2}\\
&\leqslant& \sum_{J\in \mathcal{P}_n(\{1,\ldots,p\})}
\left(\|f_{J}-Y\|_{n}^{2} + \frac{2\sigma^{2}}{n}|J| \right)
g_{J}\pi_J-\sigma^{2}.
\end{eqnarray*}
By definition of $g_{J}$ we have
\begin{align*}
\|f_{J}-Y\|_{n}^{2} + \frac{2\sigma^{2}|J|}{n}  =\ &
-\frac{1}{\lambda} \log\left(\frac{g_{J}}{\sum_{J\in
\mathcal{P}_n(\{1,\ldots,p\})}g_J \pi_J}\right)\\
&{}- \frac{1}{\lambda}
\log\left( \sum_{J\in \mathcal{P}_n(\{1,\ldots,p\})}g_J\pi_{J}
\right).
\end{align*}
Integrating the above inequality w.r.t. the probability distribution
$ \frac{1}{C}g\cdot\pi$ (where $C = \sum_{J\in
\mathcal{P}_n(\{1,\ldots,p\})}g_J \pi_J$ is the normalization
factor) and using the fact that
$$\sum_{J\in \mathcal{P}_n(\{1,\ldots,p\})}\frac{1}{C}g_{J}\pi_J\log\left(\frac{1}{C}g_{J}\right) =
K\left(\frac{g\cdot\pi}{C},\pi\right)\geqslant 0$$ as well as a
convex duality argument (cf., e.g., \cite{DZ98}, p. 264) we get
\begin{equation*}
\hat{r}_{n}(Y)\leqslant \sum_{J\in
\mathcal{P}_n(\{1,\ldots,p\})}\left(\|Y-f_{J}\|_{n}^{2} +
\frac{2\sigma^{2}}{n}|J| \right)\pi_{J}' +
\frac{1}{\lambda}K(\pi',\pi) -\sigma^{2},
\end{equation*}
for all probability measure $\pi'$ on $\mathcal{P}(\{1,\ldots,p\})$.
Taking the expectation in the last inequality we get for any $\pi'$
\begin{align*}
&\E\left[ \|\hat{f}_{n}-f\|_{n}^{2} \right]
 =\E[\hat{r}_{n}(Y)]\\
&\quad \leqslant \sum_{J\in \mathcal{P}_n(\{1,\ldots,p\})}\left( \E[\|
f_{J} - Y \|_{n}^{2}] + \frac{2\sigma^{2}}{n}|J| \right)\pi_{J}' +
\frac{1}{\lambda}K(\pi',\pi) - \sigma^{2}\\
&\quad \leqslant \sum_{J\in \mathcal{P}_n(\{1,\ldots,p\})}\left( \E[\|
f_{J} - f \|_{n}^{2}] +
\frac{2}{n}\sum_{i=1}^{n}\E[W_{i}f_{J}(X_{i},Y)] +
\frac{2\sigma^{2}}{n}|J| \right)\pi_{J}'\\
& \hspace{9cm}
+
\frac{1}{\lambda}K(\pi',\pi)\\
&\quad \leqslant \sum_{J\in \mathcal{P}_n(\{1,\ldots,p\})}\left( \E[\|
f_{J} - f \|_{n}^{2}] + \frac{4\sigma^{2}}{n}|J| \right)\pi_{J}' +
\frac{1}{\lambda}K(\pi',\pi),
\end{align*}
where we have used Stein's argument $\E[W_{i}f_{J}(X_{i},Y)] =
\sigma^{2}\E\left[
\partial_{i}f_{J}(X_{i},Y) \right]$ and the fact that $\sum_{i=1}^{n}\partial_{i}f_{J}(X_{i},Y)= 1$ in the last line. Finally taking $\pi'$ in
the set of Dirac distributions on the subset $J$ of $\{1,\ldots,p\}$
yields the theorem.
\end{proof}

\begin{proof}[Proof of Theorem \ref{thm_K2}]
  First note that any minimizer $\theta\in \mathbb R^p$ of the
  right-hand-side in (\ref{result-K}) is such that $|J(\theta)| \leqslant
  \mathrm{rank}(A)\leqslant n$ where we recall that $A = (\phi_j(X_i))_{1\leqslant i \leqslant n,1\leqslant j \leqslant
  p}$. Indeed, for any $\theta\in \mathbb R^{p}$ such that
  $|J(\theta)|>\mathrm{rank}(A)$ we can construct a vector $\theta'\in \mathbb
  R^p$ such that $f_\theta = f_{\theta'}$ and $|J(\theta')|\leqslant
  \mathrm{rank}(A)$ and the mapping $x\rightarrow
  x\log\left(\frac{ep\alpha}{x}\right)$ is nondecreasing on $(0,p]$.

Next for any $J\in \mathcal{P}_n(\{1,\ldots,p\})$ we have
\begin{eqnarray*}
  \mathbb E[\|f_J - f\|_n^2] = \min_{\theta\in \Theta(J)}\left\{\|f_\theta
- f\|_n^2\right\} +
  \frac{\sigma^2|J|}{n}= \min_{\theta\in \Theta(J)}\left\{\|f_\theta
- f\|_n^2 +
  \frac{\sigma^2|J(\theta)|}{n}\right\}.
\end{eqnarray*}
Thus
\begin{align*}
   &\min_{J\in\mathcal{P}_n(\{1,\ldots,p\})}\left\{\mathbb E[\|f_J -f\|_n^2] + \frac{1}{\lambda}\log\left(\frac{1}{\pi_J}\right) + \frac{\sigma^2J}{n}
   \right\}\\
   &\hspace{2cm}= \min_{J\in\mathcal{P}_n(\{1,\ldots,p\})}\min_{\theta \in \Theta(J)}\left\{\|f_\theta -f\|_n^2 + \frac{1}{\lambda}\log\left(\frac{1}{\pi_{J(\theta)}}\right) + \frac{\sigma^2|J(\theta)|}{n}
   \right\}\\
&\hspace{2cm}= \min_{\theta\in\mathbb R^p}\left\{\|f_\theta -f\|_n^2
+ \frac{1}{\lambda}\log\left(\frac{1}{\pi_{J(\theta)}}\right) +
\frac{\sigma^2|J(\theta)|}{n}
   \right\}.\\
\end{align*}
%where $J(\theta) = \{j\,:\, \theta_j\neq 0\}$.
Combining the above display with Proposition \ref{thm_K} and our
definition of the prior $\pi$ gives the result.
\end{proof}

%s4.2 ###
\subsection{Proof of Theorem \ref{SOI}}
\label{proofSOI}

\allowdisplaybreaks
We state below a version of Bernstein's inequality useful in the
proof of Theorem \ref{SOI}. See Proposition 2.9 page 24 in
\cite{massart}, more precisely Inequality (2.21).

\begin{lemma}
\label{lemmemassart} Let $T_{1}$, \ldots, $T_{n}$ be independent real
valued random variables. Let us assume that there is two constants
$v$ and $w$ such that
$$ \sum_{i=1}^{n} \mathbb{E}[T_{i}^{2}] \leq v $$
and for all integers $k\geq 3$,
$$ \sum_{i=1}^{n} \mathbb{E}\left[(T_{i})_{+}^{k}\right] \leq v\frac{k!w^{k-2}}{2}. $$
Then, for any $\zeta\in (0,1/w)$,
$$ \mathbb{E}
\exp\left[\zeta\sum_{i=1}^{n}\left[T_{i}-\mathbb{E}(T_{i})\right]
\right]
        \leq \exp\left(\frac{v\zeta^{2}}{2(1-w\zeta)} \right) .$$
\end{lemma}

\begin{proof}[Proof of Theorem \ref{SOI}]
For any $\theta \in \Theta_{K+1}$ define the random variables
$$ T_{i} =T_i(\theta)=  - \left(Y_{i}-f_{\theta}(X_{i})\right)^{2}
                    + \left(Y_{i}-f_{\bar{\theta}}(X_{i})\right)^{2} .$$
Note that these variables are independent. We have
\begin{multline*}
\sum_{i=1}^{n} \mathbb{E}[T_{i}^{2}] = \sum_{i=1}^{n} \mathbb{E}
\left[
       \left[2Y_{i} - f_{\bar{\theta}}(X_{i})-f_{\theta}(X_i)\right]^{2}
\left[f_{\bar{\theta}}(X_{i})-f_{\theta}(X_i)\right]^2
            \right]
\\
= \sum_{i=1}^{n} \mathbb{E} \left[
       \left[2W_{i} +2f(X_i) - f_{\bar{\theta}}(X_{i}) - f_{\theta}(X_{i})\right]^{2}
\left[f_{\bar{\theta}}(X_{i})-f_{\theta}(X_i)\right]^2
            \right]
\\
\leq \sum_{i=1}^{n} \mathbb{E} \left[
        \left[8 W_{i}^{2} + 2(2\|f\|_\infty + L(2K+1))^2\right]
\left[f_{\bar{\theta}}(X_{i})-f_{\theta}(X_i)\right]^2
            \right]
\\
= \sum_{i=1}^{n} \mathbb{E} \left[ 8 W_{i}^{2} + 2(2\|f\|_\infty +
L(2K+1))^2\right]
     \mathbb{E} \left[\left[f_{\bar{\theta}}(X_{i})-f_{\theta}(X_i)\right]^2
            \right]
\\
\leq n \left[ 8 \sigma^{2} + 2(2\|f\|_\infty + L(2K+1))^2\right]
\left[R(\theta) - R(\bar{\theta})\right]=:v(\theta,\bar{\theta})=v,
\end{multline*}
where we have used in the last line Pythagore's Theorem to prove
$\|f_{\theta} -
f_{\bar{\theta}}\|^2 = R(\theta) - R(\bar{\theta})$. Next we have,
for any integer $k\geq 3$, that
\begin{multline*}
\sum_{i=1}^{n} \mathbb{E}\left[(T_{i})_{+}^{k}\right] \leq
\sum_{i=1}^{n} \mathbb{E} \left[
       \left|2Y_{i} - f_{\bar{\theta}}(X_{i})-f_{\theta}(X_i)\right|^{k}
\left|f_{\bar{\theta}}(X_{i})-f_{\theta}(X_i)\right|^k
            \right]
\\
\leq \sum_{i=1}^{n} \mathbb{E} \left[
       2^{2k-1}\left[  |W_{i}|^{k} +(\|f\|_\infty + L(K+1/2))^{k} \right]
\left|f_{\bar{\theta}}(X_{i})-f_{\theta}(X_i)\right|^k
            \right]
\\
\leq \sum_{i=1}^{n} \mathbb{E} \left[
       2^{2k-1}\left[  |W_{i}|^{k} \!+\!(\|f\|_\infty \!+\! L(K\!+\! 1/2))^{k} \right]
       [L(2K\!+\!1)]^{k-2}
\left[f_{\bar{\theta}}(X_{i})\!-\!f_{\theta}(X_i)\right]^{2}
            \right]
\\
\leq 2^{2k-1}\left[\sigma^{2}k!\xi^{k-2}
+(\|f\|_\infty+L(K+1/2))^{k} \right]
          [L(2K+1)]^{k-2}\\
 \hspace*{5cm}{}\times         \sum_{i=1}^{n}\mathbb{E}\left[\left[f_{\bar{\theta}}(X_{i})-f_{\theta}(X_i)\right]^{2}
            \right]
\\
\leq  \frac{  (\sigma^{2}k!\xi^{k-2} +(\|f\|_\infty+L(K+1/2))^{k})(4L(2K+1))^{k-2} }{4(\sigma^{2} + (\|f\|_\infty +
L(K+1/2))^{2})}v
\\
\leq \frac{1}{4}\left( k!\xi^{k-2} + [\|f\|_\infty + L(K+1/2)]^{k-2}
\right)[4L(2K+1)]^{k-2}v
\\
\leq \frac{2}{4}k!\left( \xi + [\|f\|_{\infty} + L(K+1/2)]
\right)^{k-2}[4L(2K+1)]^{k-2v} \leq v\frac{k!w^{k-2}}{2},
\end{multline*}
with $ w:=8(\xi + [\|f\|_\infty + L(K+1/2)])L(K+1/2)$.

Next, for any $\lambda\in (0,n/w)$ and $\theta\in\Theta_{K+1}$,
applying Lemma \ref{lemmemassart} with $\zeta=\lambda/n$ gives
$$
\mathbb{E} \exp\left[\lambda
\Bigl(R(\theta)-R(\bar{\theta})-r(\theta)+r(\bar{\theta})\Bigr)\right]
\leq
\exp\left[\frac{v\lambda^{2}}{2n^{2}(1-\frac{w\lambda}{n})}\right].
$$
Set $C=8\left(\sigma^{2} + [\|f\|_\infty + L(K+1/2)]^{2}\right)$.
For the sake of simplicity let us put
%e14 ###
\begin{equation}
\label{defbeta}
\beta = \left(\lambda
-\frac{\lambda^{2}C}{2n(1-\frac{w\lambda}{n})}\right) .
\end{equation}
For any $\varepsilon
>0$ the last display yields
$$
\mathbb{E} \exp\left[\beta
                      \Bigl(R(\theta)-R(\bar{\theta})\Bigr)
                    +\lambda\Bigl(-r(\theta)+r(\bar{\theta})\Bigr)
         - \log\frac{2}{\varepsilon}\right] \leq \frac{\varepsilon}{2}.
$$
Integrating w.r.t. the probability distribution ${\rm m}(\cdot)$ we get
$$
\int \mathbb{E} \exp\Biggl[\beta
                      \Bigl(R(\theta)-R(\bar{\theta})\Bigr)
                    +\lambda\Bigl(-r(\theta)+r(\bar{\theta})\Bigr)
         - \log\frac{2}{\varepsilon}\Biggr] {\rm m}(d\theta) \leq \frac{\varepsilon}{2}.
$$
Next, Fubini's theorem gives
\begin{multline*}
\mathbb{E} \int \exp\Biggl[\beta
                      \Bigl(R(\theta)-R(\bar{\theta})\Bigr)
                    +\lambda\Bigl(-r(\theta)+r(\bar{\theta})\Bigr)
         - \log\frac{2}{\varepsilon}\Biggr] {\rm m}(d\theta)
\\
=
\mathbb{E} \int \exp\Biggl[\beta
                      \Bigl(R(\theta)-R(\bar{\theta})\Bigr)
                    +\lambda\Bigl(-r(\theta)+r(\bar{\theta})\Bigr)
                 - \log\left[\frac{d\tilde{\rho}_{\lambda}}{d{\rm m}}(\theta)\right]
         - \log\frac{2}{\varepsilon}\Biggr] \tilde{\rho}_{\lambda}(d\theta)
\\
\leq \frac{\varepsilon}{2}.
\end{multline*}
Jensen's inequality yields
$$
\mathbb{E} \exp\Biggl[\beta
                      \left(\int R d\tilde{\rho}_{\lambda}-R(\bar{\theta})\right)
                    +\lambda\left(-\int r d\tilde{\rho}_{\lambda}+r(\bar{\theta})\right)
                 - \mathcal{K}(\tilde{\rho}_{\lambda},{\rm m})
         - \log\frac{2}{\varepsilon}\Biggr] \leq \frac{\varepsilon}{2}.
$$
Now, using the basic inequality $\exp(x) \geq
\mathbf{1}_{\mathbb{R}_{+}}(x)$ we get
$$
\mathbb{P}\Biggl\{ \beta
                      \left(\int R d\tilde{\rho}_{\lambda}-R(\bar{\theta})\right)
                    +\lambda\left(-\int r d\tilde{\rho}_{\lambda}+r(\bar{\theta})\right)
                 - \mathcal{K}(\tilde{\rho}_{\lambda},{\rm m})
         - \log\frac{2}{\varepsilon}\Biggr] \geq 0
\Biggr\} \leq \frac{\varepsilon}{2}.
$$
Using Jensen's inequality again gives
$$ \int R d\tilde{\rho}_{\lambda} \geq R\left(\int\theta\tilde{\rho}_{\lambda}(d\theta)\right)
                                      = R(\tilde{\theta}_{\lambda}).$$
Combining the last two displays we obtain
\begin{equation*}
\mathbb{P}\Biggl\{ R(\tilde{\theta}_{\lambda}) - R(\bar{\theta})
\leq \frac{ \int r d\tilde{\rho}_{\lambda} - r(\bar{\theta}) +
\frac{1}{\lambda}\left[\mathcal{K}(\tilde{\rho}_{\lambda},{\rm
m})+\log\frac{2}{\varepsilon}\right] } {\frac{\beta}{\lambda} }
\Biggr\} \geq 1-\frac{\varepsilon}{2}.
\end{equation*}

Now, using Lemma 1.1.3 in Catoni \cite{manuscrit} we obtain that
%e15 ###
\begin{equation}\label{interm3bis}
\mathbb{P}\Biggl\{ R(\tilde{\theta}_{\lambda}) - R(\bar{\theta})
\leq \inf_{\rho\in\mathcal{M}_{+}^{1}(\Theta_{K+1})} \frac{ \int r
d\rho - r(\bar{\theta}) +
\frac{1}{\lambda}\left[\mathcal{K}(\rho,{\rm
m})+\log\frac{2}{\varepsilon}\right] } {\frac{\beta}{\lambda} } \Biggr\} \geq 1-\frac{\varepsilon}{2}.
\end{equation}

We now want to bound from above $r(\theta) - r(\bar{\theta})$ by
$R(\theta)- R(\bar{\theta})$. Applying Lemma \ref{lemmemassart} to
$\tilde{T}_i(\theta) = - T_i(\theta)$ and similar computations as
above yield successively
$$
\mathbb{E} \exp\left[\lambda \Bigl(R(\bar{\theta})-R(\theta) +
r(\theta) - r(\bar{\theta})\Bigr)\right] \leq
\exp\left[\frac{v\lambda^{2}}{2n^{2}(1-\frac{w\lambda}{n})}\right],
$$
and so for any (data-dependent) $\rho$,
$$
\mathbb{E} \exp\Biggl[\gamma\left(-\int Rd\rho + R(\bar{\theta})\right)
+ \lambda \left( \int r d\rho - r(\bar{\theta})\right) -
\mathcal{K}(\rho,{\rm m}) - \log \frac{2}{\varepsilon}\Biggr] \leq
\frac{\varepsilon}{2},
$$
where
%e16 ###
\begin{equation}
\label{defgamma}
\gamma = \left(\lambda
+\frac{\lambda^{2}C}{2n(1-\frac{w\lambda}{n})}\right) .
\end{equation}
Now,
%e17 ###
\begin{equation}
\label{interm4} \mathbb{P}\Biggl\{ \int rd\rho - r(\bar{\theta})
\leq \frac{\gamma}{\lambda} \left[\int
Rd\rho - R(\bar{\theta}) \right] + \frac{1}{\lambda}\left[
\mathcal{K}(\rho,\rm m) + \log \frac{2}{\varepsilon} \right]
\Biggr\}\geq 1 - \frac{\varepsilon}{2}.
\end{equation}
Combining (\ref{interm4}) and (\ref{interm3bis}) with a union bound
argument gives
\begin{multline*}
\mathbb{P}\Biggl\{ R(\tilde{\theta}_{\lambda}) - R(\bar{\theta})
\\
\leq \inf_{\rho\in\mathcal{M}_{+}^{1}(\Theta_{K+1})} \frac{ \gamma \left[\int Rd\rho -
R(\bar{\theta}) \right] + 2 \left[
\mathcal{K}(\rho,\rm m) + \log \frac{2}{\varepsilon} \right] } {
\beta  } \Biggr\} \\
\geq 1-\varepsilon,
\end{multline*}
where $\mathcal{M}_{+}^{1}(\Theta_{K+1})$ is the set of all
probability measures over $\Theta_{K+1}$.

Now for any $\delta\in (0,1]$ taking $\rho$ as the uniform
probability measure on the set $ \{ t\in\Theta(J(\bar{\theta})):
|t-\bar{\theta}|_{1} \leq \delta \}\subset
\Theta_{K+1}(J(\bar{\theta}))$ gives
\begin{multline*}
\mathbb{P}\Biggl\{ R(\tilde{\theta}_{\lambda}) \leq  R(\bar{\theta})
+ \frac{1}{\beta}  \Biggl[
\gamma L^{2}\delta^2
+ 2 \Biggl(|J(\bar{\theta})|\log\frac{K+1}{\delta}
\\
+
|J(\bar{\theta})|\log \frac{1}{\alpha}
+\log\left(\frac{1}{1-\alpha}\right)+ \log {p\choose
|J(\bar{\theta})|} +\log\frac{2}{\varepsilon}\Biggr) \Biggr]\Biggr\}
 \geq 1-\varepsilon.
\end{multline*}
Taking $\delta=n^{-1}$ and the inequality $ \log {p\choose
|J(\bar{\theta})|} \leq |J(\bar{\theta})|\log
\frac{pe}{|J(\bar{\theta})|}$ gives
\begin{multline}
\label{boundlambda} \mathbb{P}\Biggl\{ R(\tilde{\theta}_{\lambda})
\leq  R(\bar{\theta}) + \frac{1}{1-\frac{\lambda C}{2(n-w\lambda)}}
\Biggl[ \left(1+\frac{\lambda
C}{2(n-w\lambda)}\right)\frac{L^2}{n^2}
\\
+ \frac{2}{\lambda}\left(|J(\bar{\theta})|\log \left(K+1\right) +
|J(\bar{\theta})|\log
\left(\frac{epn}{\alpha|J(\bar{\theta})|}\right)
+\log\left(\frac{2}{\varepsilon(1-\alpha)}\right)\right)
\Biggr]\Biggr\} \geq 1-\varepsilon
\end{multline}
where we replaced $\gamma$ and $\beta$ by their definitions, see
\eqref{defbeta} and \eqref{defgamma}.
Taking now $\lambda= n/(2\mathcal{C}_1)$ (where we recall that
$\mathcal{C}_1 = C\vee w$) in (\ref{boundlambda}) gives
\begin{multline*}
\mathbb{P}\Biggl\{ R(\tilde{\theta}_{\lambda}) \leq R(\bar{\theta})
+ \frac{3L^2}{n^2} + \frac{8\mathcal{C}_1}{n}\Biggl[
|J(\bar{\theta})|\log\left(K+1\right)
\\
+ \left(|J(\bar{\theta})|\log\left(\frac{enp}{\alpha
|J(\bar{\theta})|} \right) +
\log\left(\frac{2}{\varepsilon(1-\alpha)}\right)\right)
\Biggr]\Biggr)\Biggr\} \geq 1-\varepsilon,
\end{multline*}
where we have used that $1-\frac{\lambda C}{2(n-w\lambda)} \geq 1/2$
and $ 1+\frac{\lambda C}{2(n-w\lambda)}\leq 3/2$.
\end{proof}

\end{document}